\documentclass[reqno,a4paper,12pt]{amsart}
\usepackage{amssymb,amsmath,amscd,amstext,amsthm,amsfonts,mathtools}
\usepackage{mathscinet,mhsetup,pcatcode,ifoption,rkeyval,textcmds}
\usepackage[]{graphicx}
\usepackage{setspace} 
\usepackage{subfig}
\usepackage[mathscr]{eucal}
\usepackage[final]{changes}
\usepackage[ansinew]{inputenc} 

\usepackage{geometry}
\usepackage{amsrefs}
\usepackage{bbm}

\numberwithin{equation}{section}

\newtheorem{theorem}{Theorem}[section]
\newtheorem{proposition}[theorem]{Proposition}

\newtheorem{corollary}[theorem]{Corollary}
\newtheorem{lemma}[theorem]{Lemma}

{\theoremstyle{definition}
{
\newtheorem{remark}[theorem]{Remark}

\newtheorem{defn}[theorem]{Definition}
}}

\newcommand{\cal}{\mathcal}

\newcommand{\A}{{\cal A}}

\newcommand{\OO}{{\cal O}}

\newcommand{\Cc}{{\mathbb{C}}}

\newcommand{\Rr}{{\mathbb{R}}}

\def\diag{\operatorname{diag}}

\def\C{\operatorname{C{}}}

\newcommand{\matbrackets}[2]{\left[\, {#1} \right]_{{#2}} }
\newcommand{\unit}{\mathbbm{1}}

\renewcommand{\d}{\mathrm d}

\newcommand{\Mat}{{\rm Mat}}

\newtheorem{rmk}[theorem]{Remark}
\newtheorem*{theorem*}{Theorem}

\newcommand{\comment}[1]{}

\newcommand{\nund}{\underline{n}}
\newcommand{\intGamma}{\Gamma^{\circ}_{\nund}}
\newcommand{\qtil}{\tilde{q}}

\begin{document}
\title[Hamiltonian Evolutionary Games]{Hamiltonian Evolutionary Games}
\date{May 2014}

\author[Alishah]{Hassan Najafi Alishah}
\address{CAMGSD\\
Departamento de Matem\'atica \\
Instituto Superior T\'ecnico\\
Av. Rovisco Pais,
1049-001 Lisboa }
\email{halishah@math.ist.utl.pt}

\author[Duarte]{Pedro Duarte}
\address{Departamento de Matem\'atica and CMAF \\
Faculdade de Ci\^encias\\
Universidade de Lisboa\\
Campo Grande, Edificio C6, Piso 2\\
1749-016 Lisboa, Portugal 
}
\email{pduarte@ptmat.fc.ul.pt}

\begin{abstract}
We introduce a class of o.d.e.'s that generalizes to polymatrix games the replicator equations on symmetric and asymmetric games. We also introduce a
new class of Poisson structures on the phase space of these systems, and  characterize the corresponding subclass of Hamiltonian polymatrix replicator systems. This extends known results for symmetric and asymmetric replicator systems. 
\end{abstract}

\maketitle
\tableofcontents


\section{Introduction}
\label{intro}
\subsection*{State of the art}
Evolutionary Game Theory (EGT) originated from the work of
 John Maynard Smith and George R. Price 
 who applied  the theory of strategic games developed by  John von Neumann and  Oskar Morgenstern to evolution problems in Biology. Unlike Game Theory, EGT investigates the dynamical processes of biological populations.

Independently A. Lotkla and V. Volterra  introduced the following  class of o.d.e.'s 
$$ \frac{d x_i}{dt} = x_i\left( \, r_i+\sum_{j=1}^n a_{ij}\,x_j\right) 
\; (1\leq i\leq n)\;,$$
currently known as Lotka-Volterra (LV) systems, and usually taken as models for the time evolution of   ecosystems in $n$ species. 
Although historically this class of systems preceded  EGT they are now considered an integral part of this theory.
The entries $a_{ij}$ represent interactions between different species, while the coefficients $r_i$ stand for the specie's natural growth rates. In his studies~\cite{MR1189803} V. Volterra gave special attention to predator-prey systems and their generalization to food chain systems in $n$ species, which fall in the category of dissipative and conservative LV systems.
Denoting by $A=\matbrackets{a_{ij}}{ij}$ its interaction matrix, a LV system is said to be {\em dissipative}, resp. {\em conservative}, if there exists
a positive diagonal matrix $D$ such that  $A D + D A^t\leq 0$,  resp. $A  D$ is skew symmetric.
The matrix $D$ was interpreted by Volterra as some sort of normalization by the average weights of the different species. If the LV system admits an equilibrium point $q\in\Rr^n$ the following function $H:{\rm int}(\Rr^n_+)\to\Rr$
$$H(x) = \sum_{j=1}^n x_j - q_j\,\log x_j $$
is either a decreasing Lyapunov function, if the system is dissipative, or else a constant of motion,  if the system is conservative. Volterra proved that  the dynamics
of any $n$ species conservative LV system can be embedded in  a Hamiltonian system of dimension $2n$.
More recently, in the 1980's, Redheffer et al. developed further the teory of dissipative LV systems, introducing and studying the class of stably dissipative systems
~\cite{MR800991,MR1027969,MR741270,MR646217,MR644963}.
In~\cite{MR1643678} a re-interpretation was given
for the Hamiltonian character of the dynamics of
any conservative LV system:
there is a Poisson structure on $\Rr^n_+$ which makes
the system Hamiltonian. Another interesting fact from~\cite{MR1643678}, which stresses the importance
of studying Hamiltonian LV systems,
 is that the limit dynamics of any
 stably dissipative LV system is described by a conservative LV system.

Another class of o.d.e.'s, which plays a central role in EGT,  
is the {\em replicator equation} defined  on the simplex $\Delta^{n-1}=\{x\in\Rr_+^n\,\vert\,\sum_{i=1}^nx_i=1\}$ by 
\begin{equation*}
\frac{d x_i}{dt} =x_i\,\left(\sum_{j=1}^n a_{ij}\, x_j -\sum_{k,j=1}^n a_{kj}\, x_k\, x_j \right)\quad\quad (1\leq i\leq n)\;.
\end{equation*}
The  coefficients of this o.d.e. are stored in  an $n\times n $ real matrix $A=\matbrackets{a_{ij}}{ij}$,
that is referred as the {\em pay-off} matrix. 
A game theoretical interpretation for this equation is provided in section~\ref{section:polymatrix}.
Check~\cite{MR1635735} on the history of this equation.
In ~\cite{MR633014} J. Hofbauer introduced a change of coordinates, mapping $\Rr_+^n$ to the simplex $\Delta^n$
minus one face, which conjugates any LV system in $\Rr^n_+$ to a time re-parametrization of a replicator system in $\Delta^n$, and vice-versa. 
Thus when a   LV system is conservative then the corresponding  replicator system is orbit equivalent to a Hamiltonian system.
On the other hand, any replicator system
on $\Delta^{n-1}$ with skew symmetric pay-off matrix
extends to a LV system on $\Rr^n_+$ with $r_i=0$,
and hence can be viewed as a restriction of a Hamiltonian LV system on $\Rr^n_+$. Up to our knowledge these are the known 
subclasses of Hamiltonian replicator systems.

Asymmetric or bimatrix games  lead to another
fundamental class of models in EGT,
the following system of o.d.e.'s
whose coefficients
are displayed in two pay-off matrices,  $A$ of order $n\times m$ and $B$ of order $m\times n$.
\begin{align*} 
\frac{d x_i}{dt} & =
x_i\,\left(\sum_{j=1}^m a_{ij}\, y_j -\sum_{k=1}^n\sum_{j=1}^m a_{kj}\, x_k\, y_j \right) \quad\quad& i=1,..,n\\
\frac{d y_j}{dt} & =
y_j\,\left(\sum_{i=1}^n  b_{ji}\, x_i - \sum_{k=1}^m \sum_{i=1}^n b_{ki}\, y_k\, x_i \right) \quad\quad &j=1,..,m\notag
\end{align*} 
The phase space of this equation is the prism
$\Delta^{n-1}\times\Delta^{m-1}$. A game theoretical interpretation is given in section~\ref{section:polymatrix}.
It was remarked by I. Eshel and E. Akin ~\cite{MR723584} that up to a time re-parametrization these systems always preserve  volume.
For $\lambda$-zero-sum games ($\lambda<0$) and $\lambda$-partnership games ($\lambda>0$), with an interior equilibrium point in the prism $\Delta^{n-1}\times\Delta^{m-1}$,
J. Hofbauer proved in~\cite{MR1393843} that this bimatrix system is
orbit equivalent to a  Hamiltonian system w.r.t. some Poisson structure in the interior of the prism.
Previously, E. Akin and V. Losert~\cite{MR765812}
had noticed the Hamiltonian character of this model in the zero-sum case.

Polymatrix games, like $n$-player games,
 generalize the concept of bimatrix games. 
  The main difference between them is that  interactions between players are bilateral in the former game but not in the latter.
The first reference we could find on the existence of equilibria for these games is
the paper of J. Howson~\cite{MR0392000} who attributes
the concept of polymatrix game to   E. Yanovskaya (1968).
More recently, the structure of Nash equilibria for polymatrix games is studied by L. Quintas in~\cite{MR1024957}.

\bigskip

\subsection*{Main results}
We introduce a class of o.d.e's, referred as {\em polymatrix replicator equation},
that generalizes to polymatrix games the symmetric and asymmetric replicator equations. 
We are not aware of any reference on this equation in the literature. 
The phase space of these systems 
are finite products of simplexes. 
We introduce the concept of {\em conservative}
polymatrix game, which in the case of bimatrix games
extends the $\lambda$-zero-sum games ($\lambda<0$) and the $\lambda$-partnership games ($\lambda>0$).
In Theorem~\ref{prop:main} we introduce a class of Poisson structures on  finite products of simplexes (see~\eqref{poissonpmg}).
We will show that these prisms are stratified Poisson spaces (see section~\ref{poissonreduction}).
Then in Theorem~\ref{conservative:hamiltonian} we show
that any conservative polymatrix game determines a Hamiltonian
 polymatrix replicator.
This work extends and unifies several known facts on Hamiltonian
replicator o.d.e.'s.
In the end of section \ref{section:polymatrix}  we compare our results with known facts mentioned in the state of the art subsection.

\bigskip

The paper is organized as follows.
In section 2 we introduce the needed concepts from
Poisson geometry. In section 3 we state and prove the
main results. In section 4 we discuss a method introduced in~\cite{MR2555841},
called  {\em singular Poisson reduction}, which gives a geometric interpretation of the Poisson structures defined in section 3.
In the last section we workout a couple of examples.

\section{Generalities on Poisson Structures} 
\label{poisson}
In this section we will provide a short introduction to Poisson geometry 
focused on  some dynamical aspects, see any standard textbook on Poisson manifolds and related topics, for example \cite{MR2178041,MR1048350}.

Let $M$ be an $n$-dimensional smooth manifold.
We denote by $C^\infty(M)$ the space of smooth functions on $M$.
A Poisson structure on $M$ is an $\Rr$-bilinear bracket $\{.,.\}:C^\infty(M)\times C^\infty(M)\rightarrow C^\infty(M)$  which satisfies:

\begin{itemize}
\item[i)] Anti-symmetry i.e. $\{f,g\}=-\{g,f\}$ for every $f,g\in C^\infty(M)$.
\item[ii)] Leibniz's rule i.e. $\{fg,h\}=f\{g,h\}+g\{f,h\}$ for every $f,g,h\in C^\infty(M)$.
\item[iii)] Jacobi identity i.e. $\{\{f,g\},h\}+\{\{g,h\},f\}+\{\{h,f\},g\}=0$ 
\end{itemize}
The Leibniz's rule  says  that for any smooth function $H:M\rightarrow\Rr$ the map $\{.,H\}: f\mapsto\{f,H\}$ is a deriviation on $C^\infty(M)$ which in turn yields a vector $X_H$ on $M$ defined by the equality $\{f,H\}=\d f(X_H)$. The vector field $X_H$ is called the {\em Hamiltonian vector field} associated to $H$ on the Poisson manifold $M$.

The singular distribution 
$D(x):=\{X_f(x)\,|\,f\in C^\infty(M)\}$
is called the {\em characteristic distribution} of $M$. 
As a consequence of the Jacobi  identity this distribution integrates to a singular foliation.
Denote by $S_x$ the leaf of this foliation through a point $x$.
The Poisson structure induces a symplectic form on each leaf $S_x$, 
passing through arbitrary point $x\in M$, of this foliation
 defined by $\omega_{S_x}(X_f,X_h)=\{f,h\}$.  The foliation $\mathcal{S}:=\{(S_x,\omega_{S_x})| x\in M\}$ is called \added{the} {\em symplectic foliation }  of the Poisson manifold $M$.
\begin{rmk}\label{rmk:poisson}
The following are well known properties of Poisson structures:
\begin{itemize}
\item[1)]  By $(i)$,
$\d H(X_H)=\{H,H\}=-\{H,H\}=0$. Thus $H$ is an integral of motion for the vector field $X_H$.
\item[2)] The dimension of the linear subspace $D(x)$ is called the {\em  rank} of the Poisson structure at point $x$,
 which is equal to the dimension of the leaf $S_x$. Since this leaf is a symplectic manifold on its own it has even dimension. 
\item[3)] The symplectic foliation $\mathcal{S}:=\{(S_x,\omega_{S_x})| x\in M\}$ completely determines the Poisson structure.
\item[4)] By definition, it is clear that every symplectic leaf $S_x$ is an invariant submanifold for any Hamiltonian vector filed $X_H$. In fact, the 
restriction of $X_H$ to $S_x$  is Hamiltonian with respect to the symplectic structure $\omega_{s_x}$. 
\item[5)] Every symplectic manifold $(N,\omega)$ is a Poisson manifold with Poisson bracket defined by $\{f,g\}_N:=\omega(X_f,X_g)$, where $X_f$ and $X_g$ are \added{the} Hamiltonian vector fields associated to $f$ and $g$ by symplectic structure. 
\item[6)] A function $f$ is called {\em Casimir } if $\{.,f\}=0$. Note that Casimirs are constants of motion for any Hamiltonian vector field. Furthermore, if  $f_1,f_2$ are two Casimirs then $\{f_1,f_2\}$ is also a Casimir due to Jacobi identity.
\end{itemize} 
\end{rmk}
In a local coordinate chart $(U,x_1,..,x_n)$, or equivalently  when $M=\Rr^n$,  a Poisson bracket takes the form
\[\{f,g\}(x)=(\d_x f)^t
\matbrackets{\pi_{ij}(x)}{ij} \d_x g=\sum_{i<j}\pi_{ij}(x)\left( \frac{\partial f}{\partial x_i}\frac{\partial g}{\partial x_j}-\frac{\partial f}{\partial x_j}\frac{\partial g}{\partial x_i}\right),\]
where 
 $\pi(x)=\matbrackets{\pi_{ij}(x)}{ij}= \matbrackets{\{x_i,x_j\}(x)}{ij}$  is 
a skew symmetric matrix valued smooth function, and for every function $f$ we 
write
$$
\d_x f=\begin{pmatrix}\frac{\partial f}{\partial x_1}(x)\\\vdots\\\frac{\partial f}{\partial x_n}(x)\end{pmatrix}\;.
$$
The Jacobi identity translates to:
\begin{equation}\label{condition:poisson}
\sum_{l=1}^{n}  \frac{\partial\pi_{ij}}{\partial x_l}\pi_{lk}+\frac{\partial\pi_{jk}}{\partial x_l}\pi_{li}+\frac{\partial\pi_{ki}}{\partial x_l}\pi_{lj} =0\quad\quad \forall i,j,k\;,
\end{equation}
or equivalently 
\begin{equation}\label{condition:poisson1}
\{\{x_i,x_j\},x_k\}+\{\{x_j,x_k\},x_i\}+\{\{x_j,x_k\},x_j\}=0\quad\quad \forall i,j,k\;.
\end{equation}
Clearly, every skew symmetric
matrix valued function 
$\pi:\Rr^n\to \Mat_{n\times n}(\Rr)$    satisfying condition~\eqref{condition:poisson} defines a Poisson structure on $\Rr^n$. In the next section we shall  introduce our Poisson structures through their associated skew symmetric matrix valued functions,
referred as  bivectors  $\pi:\Rr^n\to \Mat_{n\times n}(\Rr)$. 
The term bivector means that   $\pi(x)$ is as a linear operator $\pi(x):(\Rr^n)^\ast\to \Rr^n$.

\begin{rmk}\label{rmk:poissonlocal}
Regarding the function $\pi$ we have 
\begin{itemize}
\item[1)] For any function $H$ the associated Hamiltonian vector field is defined by 
\[X_H=\pi\,\d H,\]
\item[2)] The  characteristic distribution $D_\pi(x)$ is the one generated by the columns of 
the matrix $\pi(x)$. 
\item[3)] It transforms under a change of variable $\psi:\Rr^n\to\Rr^n$ by
\begin{equation}\label{Poissoncondition1}
(\d_m \psi)\pi (m)(\d_m \psi)^t=\pi(\psi(m)),
\end{equation}
\item[4)] A function $f\in C^\infty(\Rr^n)$ is a Casimir if 
\[X_f= \pi\,\d f=0.\]
\end{itemize}
\end{rmk}

Let $(M,\{,\}_M)$ and $(N,\{.,.\}_N)$ be two Poisson manifolds. 
\begin{defn}\label{def:poissonmap}
A smooth map $\psi:M\to N$ will be called a {\em Poisson map} if and only if 
\begin{equation*}
\{f\circ\psi,h\circ\psi\}_M=\{f,h\}_N\circ\psi\quad \forall f,h\in C^\infty(N).
\end{equation*}
In local coordinate, this condition reads as
\begin{equation}\label{Poissoncondition2}
(\d_m \psi)\pi_M (m)(\d_m \psi)^t=\pi_N(\psi(m)),
\end{equation}
where $\pi_M $ and $\pi_N$ are skew symmetric matrix valued functions associated to Poisson structures of $M$ and $N$, respectively,
 and $\d_m\psi$ is the Jacobian matrix of the map $\psi$ at point $m$.  
\end{defn}


\section{Polymatrix games}
\label{pmg}

\label{section:polymatrix}
In this section we introduce the {\em evolutionary polymatrix games} to which our main result applies.
This class of systems contains both the {\em replicator models} and the {\em evolutionary bimatrix games}.

Consider a population whose individuals interact with each other using one of $n$ possible
pure strategies. The state of the population is described by a probability vector $p=(p_1,\ldots, p_n)$, 
with the usage frequency of each pure strategy. This vector is a point in 
the {\em $n-1$-dimensional simplex}  
$$\Delta^{n-1}=\{\, (x_1,\ldots, x_n)\in\Rr^ n\,:\, 
x_1+\ldots + x_n=1,\, x_i\geq 0\, \}\;.$$
A symmetric game is specified by a $n\times n$  pay-off matrix $A=\matbrackets{a_{ij}}{ij}$,
where the entry $a_{ij}$ represents the pay-off  
of an individual using pure strategy $i$ against another using pure strategy $j$.
Given   $x\in\Delta^ {n-1}$, the value
$(A\,x)_i=\sum_{j=1}^n a_{ij}\, x_j$ represents the average pay-off of strategy $i$
within a population at state $x$. Similarly, the value
$x^ t\, A\, x =\sum_{i,j=1}^n a_{ij}\, x_i\,x_j$ stands for the overall average of a population at state $x$,
while the difference $(A\,x)_i - x^ t\, A\, x$ measures the {\em relative fitness} of strategy $i$ in the population $x$.
The {\em replicator} model is the following o.d.e. on $\Delta^{n-1}$
\begin{equation}\label{replicator}
\frac{d x_i}{dt}  = x_i\,\left( (A\,x)_i - x^ t\, A\, x \right) \quad 1\leq i\leq n
\end{equation}
which says that the logarithmic growth rate of each pure strategy's frequency  equals its relative fitness.
The flow of this  o.d.e. is complete and leaves the simplex $\Delta^ {n-1}$ invariant, as well as every of its faces.

Next we introduce the class of evolutionary asymmetric, or bimatrix games, where two
groups of individuals within a population (e.g. males and females), or two different populations, interact using
different sets of strategies, say $n$ strategies for the first group and $m$ strategies for the second.
The state of this model is a pair of probability vectors in the $(n+m-2)$-dimensional prism $\Gamma_{n,m}=\Delta^{n-1}\times\Delta^{m-1}$. There are no interactions within each group. The game is specified by two pay-off matrices: a $n\times m$ matrix  $A=\matbrackets{a_{ij}}{ij}$,
where $a_{ij}$ is the pay-off for a member of the first group using strategy $i$ against
an individual  of the second group using strategy $j$,
 and a $m\times n$ matrix $B=\matbrackets{b_{ij}}{ij}$ 
with  the pay-offs for  the second group members.
Assuming the first and second group states are $x$ and $y$, respectively,
the value $(A\,y)_i$ is the average pay-off for a first group individual using strategy $i$,
the number $x^ t\, A\, y$ is the overall average
pay-off for the first group members, and the difference $(A\,y)_i-x^ t\, A\, y$ measures the {\em relative fitness} of the first group 
strategy $i$. Similarly,
$(B\,x)_j-y^ t\, B\, x$ measures the {\em relative fitness} of the second group 
strategy $j$ when the group  states are  $x$ and $y$.
The  bimatrix replicator is the following o.d.e. on the prism $\Gamma_{n,m}$
 \begin{align}\label{bimatrix}
\frac{d x_i}{dt}  &= x_i\,\left( (A\,y)_i - x^ t\, A\, y \right) \quad 1\leq i\leq n\\
\frac{d y_j}{dt}  &= y_j\,\left( (B\,x)_j - y^ t\, B\, x \right) \quad 1\leq j\leq m \nonumber
\end{align}
which again says  that the logarithmic growth rate of each  strategy's frequency  equals its relative fitness.
The flow of this  o.d.e. is complete and leaves the  prism $\Gamma_{n,m}$ invariant, as well as every of its faces.

\bigskip

Finally we introduce the class of polymatrix replicators.
Consider $p$ different populations, or else a single population stratified in $p$  groups.
We shall use greek letters like $\alpha$ and $\beta$ to denote these groups.
Assume that for each group $\alpha \in\{1,\ldots, p\}$, there are
 $n_\alpha$ pure strategies for interacting with members of
another group, including its own. Let us call
{\em signature} of the game to the vector $\nund=(n_1,\ldots, n_p)$.
The total number of strategies is therefore $n=n_1+\ldots + n_p$.
The polymatrix game is specified by a single $n\times n$ matrix 
$A=\matbrackets{a_{ij}}{ij}$ with the pay-off $a_{ij}$ for a user of strategy $i$, member of one group,
against a user of strategy $j$, member of another group, possibly the same.
The main difference between polymatrix games and the symmetric game, also specified by a single matrix $A$,
is that in the polymatrix game competition is restricted to members of the same group. 
This means that the relative fitness of each strategy
refers to the overall average pay-off of strategies within the same group.
To be more precise we need to introduce some notation.
We decompose $A$ in blocks, $A=\matbrackets{A^{\alpha,\beta}}{\alpha,\beta}$,
where each block $A^{\alpha,\beta}=\matbrackets{a^{\alpha,\beta}_{ij}}{ij}$
is a $n_\alpha\times n_\beta$ matrix.
Similarly we decompose each vector $x\in\Rr^n$ as
$x=(x^ \alpha)_\alpha$, where $x^{\alpha}\in\Rr^{n_\alpha}$.
We say that a strategy $i$ belongs to a group $\alpha$, and write $i\in \alpha$,
if and only if  $ n_1+\ldots + n_{\alpha-1} <i \leq n_1+\ldots + n_{\alpha}$.
Similarly we write $(i,j)\in\alpha\times\beta$ when $i\in\alpha$ and $j\in\beta$.
With this notation we have
\begin{enumerate}
\item[(a)] $x^{\alpha}_i=x_i$ \, if $i\in\alpha$, and
\item[(b)] $a^{\alpha,\beta}_{ij}=a_{ij}$ \, if $(i,j)\in  \alpha\times\beta$.
\end{enumerate}
Hence the difference $(A\,x)_i - \sum_{\beta=1}^p  (x^{\alpha})^ t A^{\alpha,\beta} x^{\beta} $
represents the {\em relative fitness} of a strategy $i\in\alpha$  within the group $\alpha$.
The polymatrix replicator is the o.d.e.
\begin{equation}\label{ode:pmg}
\frac{ d x^{\alpha}_i}{dt} = x^{\alpha}_i\,\left( (A\,x)_i - \sum_{\beta=1}^p (x^{\alpha})^ t A^{\alpha,\beta} x^{\beta} \right) \quad \forall\; i\in \alpha,\; \alpha\in\{1,\ldots, p\}\;,
\end{equation}
which once more says that the logarithmic growth rate of each pure strategy's frequency  equals its relative fitness.
The flow of this  o.d.e. is complete and leaves the prism $\Gamma_{\nund}=\Delta^{n_1-1}\times \ldots \times \Delta^{n_p-1}$ invariant. The underlying vector field on $\Gamma_{\nund}$ will be denoted by $X_A$.
The pair $G=(\nund,A)$ will be referred as a
{\em polymatrix game}, and the dynamical system determined by $X_A=X_{(\nund,A)}$ as the associated {\em polymatrix replicator} on $\Gamma_{\nund}$.

\bigskip

\begin{rmk}
When $p=1$, $\Gamma_{\nund}=\Delta^{n-1}$ and the evolutionary polymatrix game~\eqref{ode:pmg}
coincides with the replicator o.d.e.~ \eqref{replicator}.
\end{rmk}

\begin{rmk}
When $p=2$ and $A^{1,1}=0$, $A^{2,2}=0$ system~\eqref{ode:pmg}
coincides with the bimatrix replicator~ \eqref{bimatrix}  on  $\Gamma_{\nund}=\Delta^{n_1-1}\times \Delta^{n_2-1}$.
\end{rmk}

\bigskip

The proofs of the following three propositions are easy exercises.

\begin{proposition}[Identity]
The correspondence $A\mapsto X_{(\nund,A)}$ is linear and its kernel is formed by matrices $A\in\Mat_{n\times n}(\Rr)$ such that the block matrix
$A^{\alpha,\beta}$ has equal rows for all $\alpha,\beta=1,\ldots, p$.
Thus, two matrices $A,B\in\Mat_{n\times n}(\Rr)$ determine the same 
vector field $X_{(\nund,A)}=X_{(\nund,B)}$ on $\Gamma_{\nund}$ iff   the block matrix
$A^{\alpha,\beta}-B^{\alpha,\beta}$ has equal rows for all $\alpha,\beta=1,\ldots, p$.
\end{proposition}

\begin{defn}\label{equivalence} Given a signature $\nund=(n_1,\ldots, n_p)$ and matrices $A,B\in\Mat_{n\times n}(\Rr)$,
we say that the polymatrix games $(\nund,A)$ and $(\nund,B)$  are equivalent,  and  write $(\nund,A) \sim (\nund,B)$, if and only if 
$A^{\alpha,\beta}-B^{\alpha,\beta}$ has equal rows for all $\alpha,\beta=1,\ldots, p$.
\end{defn}

Equivalent matrices determine the same
evolutionary polymatrix game on $\Gamma_{\nund}$.
In other words  $(\nund,A)\sim(\nund,B)$\, iff\, $X_{(\nund,A)}=X_{(\nund,B)}$.

\begin{proposition}[Equilibria]
A point $q\in\Gamma_{\nund}$ is an equilibrium of $X_{(\nund,A)}$ if and only if $(A q)_i=(A q)_j$
for all $\alpha=1,\ldots, p$ and every $i,j\in\alpha$.
\end{proposition}

\begin{defn}
Given a signature $\nund=(n_1,\ldots, n_p)$,
we define the set
$$\mathscr{I}_{\nund}:=\{\,I\subset \{1,\ldots, n\}\,:\,
\# (I\cap \alpha)\geq 1,\; \forall\,\alpha=1,\ldots, p\,\}\;,$$
where
$I\cap \alpha:= I\cap [n_1+\ldots + n_{\alpha-1}+1, n_1+\ldots + n_{\alpha}]$.
A set $I\in \mathscr{I}_{\nund}$ determines the face $\sigma_I:=\{\,x\in\Gamma_{\nund}\,:\,
x_j=0,\, \forall\, j\notin I\,\}$ of $\Gamma_{\nund}$.
\end{defn}

The correspondence between sets in $\mathscr{I}_{\nund}$
and faces of $\Gamma_{\nund}$ is bijective.

\begin{defn} \label{restricted:pmg}
Consider a polymatrix game $G=(\nund,A)$.
Given a set $I\in\mathscr{I}_{\nund}$ the pair $G\vert_I=(\nund^I,A_I)$,
where
$\nund^I=(n_1^I,\ldots, n_p^I)$ with
$n_\alpha^I=\#(I\cap \alpha)$,  and  $A_I=\matbrackets{a_{ij}}{i,j\in I}$,
is called the {\em restriction} of the polymatrix game $G$ to the face $I$.
\end{defn}

The following proposition says that the restriction of a polymatrix replicator to a face
is another polymatrix replicator.

\begin{proposition}[Inheritance]
Consider the system~\eqref{ode:pmg} associated to the polymatrix game $G=(\nund,A)$. 
Given $I\in\mathscr{I}_{\nund}$, the face $\sigma_I$ of $\Gamma_{\nund}$  is invariant under the flow of $X_{(\nund,A)}$ and the restriction of~\eqref{ode:pmg} to $\sigma_I$ is the polymatrix replicator associated to the restricted game $G\vert_I$.
\end{proposition}

\bigskip

We set some notation in order to produce neater formulas. 
In any matrix equality the vectors in $\Rr^n$, or $\Rr^{n_\alpha}$, should be identified with column matrices. We set
$\unit=\unit_{n} =(1,1,..,1)\in\Rr^{n}$ and will omit the subscript $n$
whenever the dimension of this vector is clear from the context.
Similarly, we write $I=I_n$ for the $n\times n$ identity matrix,
and we omit the subscript $n$
whenever its value is clear.
Given $x\in \Rr^n$, we denote by $D_x$ the $n\times n$ diagonal matrix
$D_x= \diag(x_i)_i$
For each $\alpha\in\{1,\ldots, p\}$ we define the $n_\alpha\times n_\alpha$ matrix
$$ T^\alpha_x:= x^\alpha\, \unit^t -I\;,$$
and set $T_x$ to be the $n\times n$ block diagonal matrix
$T_x=\diag(T^\alpha_x)_\alpha$.

\bigskip
 
Given a polymatrix game $G=(\nund,A)$, we 
define the matrix valued  mapping $\pi_A:\Rr^n\to \Mat_{n\times n}(\Rr)$  
\begin{equation}\label{poissonpmg}
\pi_A(x):=(-1)\, T_x\, D_x \, A\, D_x\,T_x^t.
\end{equation}

We have
  $D_x\,A\, D_x=\matbrackets{D_{x^\alpha}\, A^{\alpha,\beta}\, D_{x^\beta}}{\alpha,\beta}$ where
  $D_{x^\alpha}\, A^{\alpha,\beta}\, D_{x^\beta} = \matbrackets{a^{\alpha,\beta}_{ij}x^{\alpha}_ix^{\beta}_j}{i\in\alpha,j\in\beta}$.
Simple calculations show  that $\pi_A(x)=\matbrackets{\pi_{A,i j}(x)}{i,j}$ where 
  for all $(i,j)\in\alpha\times\beta$
\begin{equation}\label{pi:ij}
\pi_{A,i j}(x)=x_i^{\alpha}x_j^{\beta}\left(-a^{\alpha,\beta}_{ij}+(A^{\alpha,\beta}x^{\beta})_i+( (A^{\alpha,\beta})^ t\, x^{\alpha})_j- (x^{\alpha})^t A^{\alpha,\beta}x^{\beta}\right).
\end{equation}
These computations reduce to the simple case $p=1,\,\,n_1=n$ where
\[\pi_A(x)=(-1)\,(x\,\unit^t -I)\, D_x\, A\, D_x\,(\unit\,x^t -I) \;,\]
and
\[\pi_{A,ij}(x)=x_ix_j(-a_{ij}+(Ax)_i+(A^tx)_j-x^tAx)\;. \]

\begin{rmk}
Notice that  $\pi_A(x)$ is a skew symmetric matrix valued map whenever $A$ is a skew symmetric matrix. 
\end{rmk}

\bigskip

\begin{defn}\label{ }
A {\em formal equilibrium} of a polymatrix game  $G=(\nund,A)$ is any vector $q\in\Rr^n$ such that
\begin{enumerate}
 \item[(a)]  $(A\,q)_i=(A \, q)_j$ for all $i,j\in\alpha$, and all $\alpha=1,\ldots, p$,
 \item[(b)]  $\sum_{j\in\alpha} q_j=1$ for all $\alpha=1,\ldots, p$.
\end{enumerate}
\end{defn}

\begin{remark}
A formal equilibrium of  $G=(\nund,A)$ is an equilibrium of the
natural extention of $X_{(\nund,A)}$ to the affine subspace 
spanned by $\Gamma_{\nund}$.
\end{remark}

Next proposition says that the existence of a formal equilibrium is a sufficient condition for the vector field $X_{(\nund,A)}$ of system~\eqref{ode:pmg}
to be a gradient of a simple function $H$ with respect to  $\pi_A$.  We denote by $\intGamma$ the topological interior of $\Gamma_{\nund}$ in the affine subspace of $\Rr^ n$ spanned by $\Gamma_{\nund}$.

\begin{proposition}\label{lemma:matrixform}
Given $A\in\Mat_{n\times n}(\Rr)$, assume there exists a formal equilibrium $q\in\Rr^n$ of $G=(\nund,A)$.
 Then, setting $H(x) =\sum_{i=1}^n q_i\log x_i$,
$$ X_{(\nund,A)}(x)=\pi_A(x)\,\d_x H  \; \text{ for every }\; x\in\intGamma\;.$$
\end{proposition}

\begin{proof}
Consider the vector field $Z= \pi_A \,\d H$.
For any $\alpha$, and $i\in\alpha$,
denote by $Z^\alpha_i(x)$ the $i$-th component of $Z(x)$. 
Using that $\sum_{j\in\beta} q^{\beta}_j=1$ we have
\begin{align*}
& Z^ \alpha_{i}(x) =\left(\sum_{\beta =1}^k  \pi^{\alpha,\beta}_A(x) \, \frac{q^ \beta}{x^ \beta} \right)_i
= \sum_{\beta =1}^k \left(\sum_{j\in\beta} \pi^{\alpha,\beta}_{A,ij}(x) \, \frac{q_j^ \beta}{x_j^ \beta} \right)_i \\
&=x_i^{\alpha} \sum_{\beta =1}^k\left[ 
( (A^{\alpha,\beta}x^{\beta})_i- (x^{\alpha})^t A^{\alpha,\beta}x^{\beta} )(\sum_{j\in\beta} q_j^\beta) + (-A^{\alpha,\beta}q^{\beta})_i+ (x^{\alpha})^t A^{\alpha,\beta}q^{\beta}  \right]\\
&=x_i^{\alpha}  \left[ (A x)_i- \sum_{\beta =1}^k(x^{\alpha})^t A^{\alpha,\beta}x^{\beta}
 +  \overbrace{  (- A q)_i + \sum_{i\in\alpha}  x^{\alpha}_i (A q)_i}^ {=0} \right]\\
&=x_i^{\alpha} \left[ (A x)_i- \sum_{\beta =1}^k(x^{\alpha})^t A^{\alpha,\beta}x^{\beta}  \right]
=X_{A,ij}^ \alpha(x)\;,
\end{align*}
where the vanishing term  follows from $q$ being an equilibrium point
and $x^{\alpha}\in\Delta^{n_\alpha-1}$. This completes the proof.
\end{proof}

For every $\alpha=1,\ldots,p$  consider the $(n_\alpha-1)\times n_\alpha$ matrix 
\[E_\alpha:= \left[\begin{array}{rrrrr}
-1 & 0 & \cdots & 0 & 1 \\
0 & -1 & \cdots & 0 & 1 \\
\vdots & \vdots & \ddots &  \vdots & \vdots \\
0 & 0 & \cdots &  -1 & 1 
\end{array}
\right] \]
and set
\begin{align}  
E&:=\mathrm{diag}(E_1,\ldots ,E_p),\nonumber \\
B&:=(-1)EAE^t \;.\label{B:def}
\end{align}
Note that $E\in\Mat_{(n-p)\times n}(\Rr)$ and $B\in\Mat_{(n-p)\times(n-p)}(\Rr)$. 
Next we introduce a mapping  $\phi:\Rr^{n-p}\to \intGamma$.
We write a vector $u\in \Rr^{n-p}=\Rr^ {n_1-1}\times \ldots \times \Rr^ {n_p-1}$
as $u=(u^\alpha)_\alpha$, where
$u^\alpha:=(u^\alpha_1,\ldots,u^\alpha_{{\scriptscriptstyle n_\alpha-1}})$,
and the components of $\phi$ as $\phi(u^ \alpha)_\alpha := (\phi^\alpha(u^ \alpha))_\alpha$, where  each 
$\phi^\alpha:\Rr^{n_\alpha-1}\to (\Delta^{n_\alpha-1})^ {\circ}$ is the map defined by
$$ \phi^\alpha(u^ \alpha):= \left(\frac{e^{u_1^\alpha}}{1+\sum\limits_{i=1}^{{n_\alpha-1}} e^{u_i^\alpha}},\ldots ,\frac{e^{u_{{\scriptscriptstyle n_\alpha-1}}^\alpha}}{1+\sum\limits_{i=1}^{{n_\alpha-1}} e^{u_i^\alpha}},\frac{1}{1+\sum\limits_{i=1}^{{n_\alpha-1}} e^{u_i^\alpha}}\right)	\;. $$

\bigskip

The following is our main result.  Consider a polymatrix game $G=(\nund,A)$. 

\begin{theorem}\label{prop:main}   
 If $A$ is skew symmetric then the mapping $\pi_A$ in \eqref{poissonpmg} defines a stratified Poisson structure on $\Gamma_{\nund}$. 
 Moreover the mapping $\phi:\Rr^ {n-1}\to\intGamma$ is a Poisson diffeomorphism if we endow $\Rr^{n-p}$
with the constant Poisson structure associated to the skew symmetric matrix $B$ defined in~\eqref{B:def}. 
\end{theorem}

\begin{proof} The map $\phi:\Rr^ {n-1}\to\intGamma$ is a diffeomorphism whose inverse is easily computed.
If $A$ is skew symmetric then so is $B$. Hence this matrix induces a constant Poisson structure
on $\Rr^{n-p}$. We want to prove that $\pi_A$ determines a Poisson structure on $\intGamma$ which makes $\phi$ a Poisson map.
By \eqref{Poissoncondition1} we just need to show that for every $u\in\Rr^{n-p}$ and $x=\phi(u)$, 
\begin{equation}\label{poissonchart}
(\d_u\phi)B(\d_u\phi)^t=(-1)T_xD_xAD_xT^t_x=\pi_A(x)\;.
\end{equation}
The fact that $\pi_A$ also determines a stratified Poisson structure on $\Gamma_{\nund}$, and  on $\Rr^n$, will be proved later.
See Remark~\ref{remark:poissonreduction}.
In order to prove \eqref{poissonchart}, it is enough to see that for every $x=\phi(u)$
\[(\d_u \phi)  E=T_xD_x\;.\]
Writting the components of $\phi^ \alpha$ as
$\phi^\alpha(u^ \alpha)=(\phi_1^\alpha(u^ \alpha),\ldots, \phi_{n_\alpha}^\alpha(u^ \alpha))$ we compute 
for every $i=1,\ldots, n_\alpha$ and $j =1,\ldots, n_\alpha-1$, 
$$ \frac{\partial \phi^\alpha_i}{\partial u^\alpha_j} = \delta_{ij}\,\phi^\alpha_i  
- \phi^\alpha_i  \,\phi^\alpha_j   \;. 
$$
Hence if $x=\phi(u)$, the Jacobian of $\phi$ at the point $u$ is
\[\d_u\phi=\mathrm{diag}(J_\alpha(x^ \alpha))_\alpha,\]
where for every $\alpha=1,\ldots, p$,
$$
J_\alpha(x_1,\ldots, x_{n_\alpha}):=\left[ \begin{array}{rrrrr}
x_1 -x_1^2 & - x_1\,x_2 &  -x_1\,x_3 & \ldots & -x_1\,x_{n_\alpha-1} \\
-x_2\,x_1 & x_2 -x_2^2  &  -x_2\,x_3 & \ldots & -x_2\,x_{n_\alpha-1} \\
-x_3\,x_1 & -x_3\,x_2  &  x_3 -x_3^2 & \ldots & -x_3\,x_{n_\alpha-1} \\
\vdots  & \vdots & \vdots & \ddots & \vdots \\
-x_{n_\alpha-1}\,x_1 & - x_{n_\alpha-1}\,x_2& - x_{n_\alpha-1}\,x_3  & \ldots & x_{n_\alpha-1} -x_{n_\alpha-1}^2 \\
-x_{n_\alpha}\,x_1 & - x_{n_\alpha}\,x_2 & - x_{n_\alpha}\,x_3 & \ldots & -x_{n_\alpha}\,x_{n_\alpha-1} 
\end{array}
\right]
$$
A simple multiplication of matrices, using the relation $x_1 +\ldots +x_{n_\alpha} =1$, shows that $J_\alpha(x^ \alpha) E_\alpha=T_{x^\alpha}D_{x^\alpha}$ for every $\alpha=1,\ldots,p$. Therefore 
\[(\d_u \phi) E=\mathrm{diag}(J_\alpha(x^ \alpha) E_\alpha)_{\alpha}=\mathrm{diag}(T_{x^\alpha}D_{x^\alpha})_{\alpha}=T_xD_x\;,\]
which completes the proof. 
\end{proof}

The next corollary gives a complete description of the symplectic foliation of $(\intGamma,\pi_A)$. 
Two examples will be given in section \ref{examples}.

\begin{corollary}[Symplectic Foliation]\label{symplecticleaves}
The symplectic leaves of  $(\intGamma,\pi_A)$ are the images of the symplectic leaves of $(\Rr^{n-p}, B)$ under the diffeomorphism $\phi$. The symplectic leaf $S_u$ of $(\Rr^{n-p}, B)$  is  the (even dimensional)  affine subspace   through  $u$   parallel to the subspace generated by the columns of $B$. 
\end{corollary}

\begin{remark}\label{remark:poissonreduction} Given a face $I\in\mathscr{I}_{\nund}$ consider 
the payoff matrix $A_I$, see Definition~\ref{restricted:pmg}.
Applying Theorem \ref{prop:main} to any face $\sigma_I$ of $\Gamma_{\nund}$
we see that $\sigma_I$ is a Poisson manifold on its own with the Poisson structure $\pi_{A_I}$.
Moreover $(\sigma_I,\pi_{A_I})$ is the restriction of $(\Gamma_{\nund},\pi_A)$ in the sense that
the inclusion map $i:\sigma_I\to \Gamma_{\nund}$ is a Poisson map.
Hence the interiors of the faces of $\Gamma_{\nund}$, regarded as Poisson manifolds, 
give $(\Gamma_{\nund},\pi_A)$ the structure of a Poisson stratified space. In addition it will be shown that  $\pi_A$ defines a Poisson structure on $\Rr^n$.  On section \ref{poissonreduction} we provide a geometric explanation for these facts.  
\end{remark}

Proposition~\eqref{lemma:matrixform} together with Theorem~\eqref{prop:main} yields the following corollary. 

\begin{corollary}\label{cor:main}
If $A$ is skew symmetric and $q\in\Rr^n$ is a formal equilibrium
of $G=(\nund,A)$   then $X_{(\nund,A)}$ is a Hamiltonian vector field, with Hamiltonian $H(x)=\sum_{i=1}^n q_i\log x_i$, w.r.t. the
Poisson structure $\pi_A$ in $\intGamma$.
\end{corollary}

\begin{defn}\label{ }
  A polymatrix game  $G=(\nund,A)$ is said to be {\em conservative}  iff
  \begin{enumerate}
\item[(a)] $G$ has a formal equilibrium, 
\item[(b)] there are matrices $A_0, D\in\Mat_{n\times n}(\Rr)$ such that 
\begin{enumerate}
\item[(i)] $A\sim A_0 D$,
\item[(ii)] $A_0$ is a skew symmetric, 
\item[(iii)] $D=\diag(\lambda_1\,I_{n_1},\ldots, \lambda_p\, I_{n_p})$  with $\lambda_\beta\neq 0$
for every $\beta=1,\ldots, p$.
\end{enumerate}
\end{enumerate}
The matrix $A_0$ will be referred as a {\em skew symmetric model}  for $G$, and $(\lambda_1,\ldots,\lambda_p)\in(\Rr\backslash\{0\})^p$   as a {\em scaling vector}.
\end{defn}

\begin{remark}\label{skew:pmg:model}
Given a skew symmetric matrix $A_0 \in\Mat_{n\times n}(\Rr)$, a signature $\nund$
and a point $\qtil\in\Rr^n$ such that
\begin{enumerate}
 \item[(a)]  $(A_0\,\qtil)_i=(A_0 \, \qtil)_j$ for all $i,j\in\alpha$, and all $\alpha=1,\ldots, p$,
 \item[(b)]  $\sum_{j\in\alpha} \qtil_j\neq 0$ \, for all $\alpha=1,\ldots, p$,
\end{enumerate}
then $G=(\nund,A_0 D)$ is a conservative polymatrix game, where  $D=\diag(\lambda_\alpha\,I_{n_\alpha})_{\alpha}$  with $\lambda_\alpha:=\sum_{j\in\alpha} \qtil_j$, and $q=D^{-1}\qtil$ is
a formal equilibrium of $G$. 
\end{remark}

 It follows from the previous remark that any generic
 skew symmetric matrix can be taken as a model for a 
 conservative polymatrix game. More precisely,
  
\begin{proposition} 
Given a signature $\nund=(n_1,\ldots, n_p)$  with $\sum_{\alpha=1}^p n_\alpha=n$, the set  of skew symmetric matrices
$A_0\in\Mat_{n\times n}(\Rr)$ such that $G=(\nund, A_0\,D)$ is a conservative polymatrix game for some diagonal matrix $D$  
 is an open and dense subset  of the space of skew symmetric matrices.
\end{proposition}

Next theorem basically says that the replicator system~\eqref{ode:pmg} is Hamiltonian for every conservative  polymatrix game.

\begin{theorem}\label{conservative:hamiltonian}
Consider a conservative polymatrix game $G=(\nund,A)$ with
formal equilibrium $q$, skew symmetric model $A_0$ and scaling co-vector $(\lambda_1,\ldots, \lambda_p)$.
Then $X_{(\nund,A)}$ is Hamiltonian in the interior of the Poisson stratified space  $(\Gamma_{\nund},\pi_{A_0})$, with Hamiltonian function
\begin{equation}\label{constant:of:motion}
H(x) =\sum_{\beta=1}^p\lambda_\beta\sum_{j\in\beta}q^{\beta}_j\log x_j^{\beta}\;.
\end{equation}
\end{theorem}

\begin{proof}
In view of definition~\ref{equivalence} we can assume that  $A=A_0 D$. 
For every $\alpha,\beta$,
$$T^{\alpha}_x\,D_{x^\alpha}\, A_0^{\alpha,\beta}\,D_{x^\beta}\,(T^{\beta}_x)^t\lambda_\beta\,\frac{q^\beta}{x^\beta} =
T^{\alpha}_x\,D_{x^\alpha}\,  A ^{\alpha,\beta}\, D_{x^\beta} \,(T^{\beta})^t\,\frac{q^\beta}{x^\beta} \;,$$
where $q^\beta/x^\beta$ stands for the componentwise division of the vectors.
Adding up in $\beta$, and using Proposition~\ref{lemma:matrixform}, we get
$$ \pi_{A_0}(x) \,\d_x H = \pi_A(x)\,\d_x\left(\sum_{j=1}^n q_j\,\log x_j\right) = X_{(\nund,A)}(x)\;. $$
\end{proof}

In the next paragraphs we compare our results 
with previously known facts.
Given  a skew symmetric matrix $A\in\Mat_{n\times n}(\Rr)$,
since  $x^t\,A\,x=0$ for all $x\in\Rr^n$, the replicator equation~\eqref{replicator} reduces
to a Lotka-Volterra equation with growth rates $r_i=0$
\begin{equation}\label{red:repl:LV}
\frac{d x_i}{dt}=x_i\,(A\,x)_i\quad 1\leq i\leq n \;.
\end{equation} 
For any $q\in\Rr^n$ such that $A\,q=0$ the function
$H(x)=\sum_{j=1}^n x_j- q_j\,\log x_j$ is a constant of motion for \eqref{red:repl:LV}. A Poisson structure on $\Rr^n$ 
defined by the bivector $\hat{\pi}_A(x)=D_x\, A\, D_x$ was  introduced in~\cite{MR1643678}.
System~ \eqref{red:repl:LV} is Hamiltonian in the interior of $\Rr^n_+$ w.r.t. $\hat{\pi}_A$ having $H$ as Hamiltonian function.
Like $\hat{\pi}_A$ the Poisson structure $\pi_A$ introduced here 
can be extended to $\Rr^n$, but unlike $\pi_A$ the structure
$\hat{\pi}_A$ does not restrict to
a Poisson structure on the simplex $\Delta^{n-1}$.
Using the Poisson structure $\pi_A$ we can  now say,
if there exists $q\in\Rr^n$ such that $A\,q=0$ and $\sum_{j=1}^n q_j\neq 0$,  that the system
\eqref{red:repl:LV} is Hamiltonian in the interior of  the simplex $\Delta^{n-1}$. Furthermore, here we study the replicator equation itself and  not a topologically equivalent LV system.

Consider now a bimatrix game with signature $(n_1,n_2)$
and matrix
\[A=\begin{pmatrix}0&A_{12}\\A_{21}&0\end{pmatrix}\;.\]
If $\lambda>0$, resp. $\lambda<0$,
 the polymatrix game $((n_1,n_2),A)$  is conservative with scaling vector $(1,\lambda)$
if and only if  it has a formal equilibrium and 
  the bimatrix game $(A_{12},A_{21})$ is $\lambda$-zero-sum game, resp. $\lambda$-partnership game, (see definitions in section 11.2 of~ \cite{MR1635735}).
 Theorem~\ref{conservative:hamiltonian} generalizes the main result (section 5) in~\cite{MR1393843}, which says that the
 evolutionary system~\eqref{bimatrix} associated to a
 $\lambda$-zero-sum  or $\lambda$-partnership game is orbit equivalent to a bipartite Lotka-Volterra system that is Hamiltonian w.r.t. some Poisson structure. This leads to the same constant
 of motion~\eqref{constant:of:motion}, but from the work~\cite{MR1393843} we only  derive the existence of a Poisson structure in the  interior of the prism $\Delta^{n_1-1}\times\Delta^{n_2-1}$
 for which some time re-parametrization of system~\eqref{bimatrix} is Hamiltonian w.r.t. that Poisson structure.
On the other hand here we provide a Poisson structure on the full prism that makes the original system Hamiltonian in the interior of the prism.

We finish this section with an extension of the class
of Hamiltonian  polymatrix replicators.
Given $p$ smooth functions $\lambda_\alpha:\Gamma_{\nund}\to \Rr\backslash\{0\}$, $\alpha=1,\ldots, p$, consider the matrix valued smooth function $D:\Gamma_{\nund}\to \Mat_{n\times n}(\Rr)$,
$D(x)=\diag(\lambda_\alpha(x) I_{n_\alpha})_\alpha$, and the system of o.d.e.'s
\begin{equation}\label{pmg:gen}
 \frac{d x^\alpha_i}{dt} = \  x^{\alpha}_i\,\left( (A D(x) \,x)_i - \sum_{\beta=1}^p \lambda_\beta(x)\,(x^{\alpha})^ t A^{\alpha,\beta}  x^{\beta} \right) \quad \forall\; i\in \alpha,\; 1\leq \alpha \leq p  
\end{equation}
associated with the vector field $Y(x)=X_{(\nund,A D(x))}(x)$  on $\Gamma_{\nund}$.

\begin{proposition}
Let $A\in\Mat_{n\times n}(\Rr)$ be a skew symmetric matrix, $q\in\Rr^n$  a formal equilibrium of $G=(\nund,A)$, and consider the $1$-form
 $$ \xi(x) = \sum_{\alpha=1}^p\sum_{j\in\alpha}
  \lambda_{\alpha}(x) \,q_j^\alpha\,\frac{d x_j^\alpha}{x_j^\alpha}\;. $$
 Then system ~ \eqref{pmg:gen} is the gradient of the $1$-form $\xi$ w.r.t. the Poisson structure $\pi_A$ in the interior of $\Gamma_{\nund}$, i.e.,
\[Y(x) =\pi_A(x)\,\xi(x)\;. \]
System~ \eqref{pmg:gen} is Hamiltonian if the  form $\xi$ is exact, i.e., there exists a smooth function $H$ such that $\xi=\d H$. But even if $\xi$ is not exact, the dynamics of $Y$ leaves invariant the symplectic foliation of $(\intGamma,\pi_{A})$. 
\end{proposition}

\begin{proof}
The proof is similar to that of Theorem~\ref{conservative:hamiltonian}.
\end{proof}

The previous model~\eqref{pmg:gen} contains the following class of  o.d.e.'s
introduced by J. Maynard Smith as an extension of the asymmetric replicator equation~\eqref{bimatrix}. 
\begin{align}\label{aggr:bimatrix}
\frac{d x_i}{dt}  &= x_i\,\left( (A_{12}\,y)_i - x^ t\, A_{12}\, y \right)\, m_1(x,y) \quad 1\leq i\leq n\\
\frac{d y_j}{dt}  &= y_j\,\left( (A_{21}\,x)_j - y^ t\, A_{21}\, x \right)\,m_2(x,y)   \quad 1\leq j\leq m \nonumber
\end{align}
See appendix J of~\cite{smith82}, 
and system (9.1) in~\cite{MR1393843}. 
Taking
 $$ A=\left[ \begin{array}{cc}0 & A_{12}\\ A_{21} & 0 \end{array}\right]\quad\text{ and  }\quad
D(x)=\left[ \begin{array}{cc}m_2(x,y)\,I_{m} & 0\\ 0 & m_1(x,y)\,I_{n} \end{array}\right]  $$
system ~\eqref{aggr:bimatrix} reduces to ~\eqref{pmg:gen}.
Since system~\eqref{aggr:bimatrix} has a dissipative character
for certain choices of the functions $m_1(x,y)$ and $m_2(x,y)$
it would be interesting to investigate analogous properties
of system~\eqref{pmg:gen}.



\section{Singular Poisson Reduction}
\label{poissonreduction}

This section is devoted to elaborate Remark \eqref{remark:poissonreduction}. We will review the singular Poisson reduction introduced in \cite {MR2555841} and use it to show that the phase space of an evolutionary game with a skew symmetric payoff matrix is a Poisson stratified space.

A smooth {\em action} of a Lie group $G$ on the manifold $M$ is a smooth map
\begin{align*}
\A:G\times M\to M
\end{align*}
such that for every $g,h\in G$ and $m\in M$ one has $\A(gh,m)=\A(g,\A(h,m))$ and $\A(e,m)=m$, where $e$ is the identity element of $G$.  For every $g\in G$,  $\A_g$ denotes the diffeomorpism defined by $m\mapsto \A(g,m).$

The action is said to be {\em proper } if the map
\begin{align*}
G\times M\to M\times M\quad\quad (g,m)\mapsto(m, \A(g,m)),
\end{align*}
is proper. Recall that a map is called proper if the preimage of any compact subset is compact. The {\em stabilizer} of a point $m\in M$ is 
\[G_m:=\{g\in G|\,\, \A(g,m)=m\}\]
The action is called {\em free} if $G_m=\{e\}$ for every $m\in M$. The set $$\OO_m:=\{\A(g,m)|\,\ g\in G\}$$ is called the {\em orbit} passing through the point $m\in M$ and the set
\[M/G:=\{\OO_m|\,\,m\in M\}\]
is called {\em the orbit space} of the action. The map $\pi_G:M\to M/G$ sending every point $m$ to its orbit $\OO_m$ is called the projection map of the action. The orbit space $M/G$ can be given a topology by $U\subset M/G$ being open if and only if $\pi_G^{-1}(U)$ is open in $M$. With this topology, it is a Hausdurff  topological space if the action is proper. We will only consider proper actions so the orbit space shall always be a Hausdorff topological space all over this section.  

For any subgroup $H$ of $G$ the {\em $H$-isotropy type submanifold} of $M$ is
\[M_H:=\{m\in M| G_m=H\},\]
and the {\em $(H)$-orbit type submanifold} is 
\[M_{(H)}:=\{m\in M| G_m\in (H)\},\]
where $(H)$ denotes the conjugacy class of $H$ in $G$.  Notice that the action of $G$ restricts to $M_{(H)}$ and $M_{(H)}/G$ can be defined and is called {\em the $(H)$-orbit type reduced  space}.   

We recall the definition of a smooth stratified space, see~\cite{MR1738431,MR1869601}.
\begin{defn}\label{defn:smoothstara}
Let $X$ be a paracompact Hausdorff  topological space. A {\em smooth stratification} of $X$ is a locally finite partition of $X$ into locally closed connected smooth submanifolds $S_i\,\,\,(i\in I)$, called the strata  of the stratification, such that for a pair of submanifolds $S_i,S_j$ if $S_i\cap\bar{S}_j\neq0$ then $S_i\subset\bar{S}_j$.  When this happens $S_i$ is called incident to $S_j$ or a boundary piece of $S_j$. 
\end{defn}
The following proposition is a well-known result in the theory of Lie group actions, see e.g.~\cite{MR1738431,MR1869601} for the proof.
\begin{proposition}
If the action of the Lie group $G$ on $M$ is proper then the orbit space $M/G$ is a smooth stratified space. Furthermore, if the action is free then $M/G$ can be equipped with a smooth manifold structure such that  the projection map $\pi_G$ becomes a submersion.
\end{proposition}

A function $f\in C^\infty(M)$ is called {\em $G$-invariant} if and only if 
\[f\circ\A_g(m)=f(m)\quad \forall g\in G , m\in M.\]
Any $G$-invariant function reduces to a function on $M/G$ so we define:
\begin{defn}
The algebra of smooth functions on the orbit space $M/G$ is 
\[C^\infty(M/G):=\{f\in C^0(M/G)|\,\,  f\circ\pi_G\in C^\infty(M)\},\]
where $C^0(M/G)$ denotes the algebra of continuous  functions on the topological space $M/G$. 
\end{defn}

\begin{defn}
An action of $G$ on the Poisson manifold $(M,\{.,.\})$ is called {\em Poisson} if $\A_g$
is a Poisson diffeomorphism for every $g\in G$.
\end{defn}

Notice that if the action is Poisson i.e.
\[\{f\circ \A_g, h\circ\A_g\}=\{f,g\}\circ\A_g\quad\forall f,h,g\] 
then Poisson bracket of any two $G$-invariant function is again $G$-invariant.  Using this fact, a bracket can be defined on the algebra of smooth functions on $M/G$ by 
\begin{equation}\label{quotientPoisson}
 \{f,h\}_{M/G}(\OO_m)=\{f\circ\pi_G,h\circ\pi_G\}(m^\prime),
 \end{equation}
where $m^\prime$ is an arbitrary element of $\OO_m$. In the case of free proper Poisson action this bracket is a Poisson bracket on the manifold $M/G$. Clearly $\pi_G$ is a Poisson map between $(M,\{.,.\})$ and $(M/G,\{.,.\}_{M/G})$.  

It is clear that $(C^\infty(M/G), \{.,.\}_{M/G})$ is a Poisson algebra. Recall that  a Poisson  algebra is an algebra equipped with a skew symmetric bracket satisfying Leibniz's rule and Jacobi identity.  We state Theorem 2.12 of \cite{MR2555841} which will be used to show that the phase space of an evolutionary polymatrix game with a skew symmetric payoff matrix is a Poisson stratified space.
\begin{theorem}[Singular Poisson Redution]\label{thm:singular}
Let $\A:G\times M\to M$ be a proper Poisson action. Then the connected components of the orbit type reduced space $M_{(H)}/G$ form a Poisson stratification $\{S_i\}_{i\in I}$ of $(M/G,\{.,.\}_{M/G})$ i.e. a smooth stratification 
such that
\begin{itemize}
\item Each Strata $S_i\,\,\,i\in I$, is a Poisson manifold.
\item The inclusions $i:S_i\to M/G$ are Poisson maps.
\end{itemize} 
\end{theorem}

The following is, basically, the example which is presented in  \cite{MR2555841}*{Section 2.5}. Let
$$M=\Cc^{n_1}\backslash\{{\bf 0}\}\times\ldots\times\Cc^{n_p}\backslash\{{\bf 0}\},$$
 where $n_1,\ldots,n_p$ are integers such that $n=n_1+\ldots+n_p$. We will consider  $M$ as a real $2n$ dimensional manifold with coordinates $(\xi,\eta)\in\Rr^{2n}$ where $z_i=\xi_i+i\eta_i$ for $i=1,\ldots,n$.  Equip $M$ with the quadratic Poisson structure defined by:
 \[\pi_M(\d w_i,\d w_j)=\{w_i,w_j\}_M={1\over 4}a_{ij}w_i w_j,\]
 where $i,j=1,\ldots,n$ for $w=\xi,\eta$ and $A$ is a skew symmetric matrix.  In the language of bivectors:
 \[\pi_M(\xi,\eta)={1\over 4}\begin{pmatrix}D_\xi AD_\xi&D_\xi AD_\eta\\D_\eta AD_\xi&D_\eta AD_\eta\end{pmatrix}.\]
 
 We shall denote by $\Cc^\ast$ the group of non zero complex numbers.  The group $(\Cc^\ast)^n$ acts on $M$ by component-wise multiplication. Denote this action by
 \begin{equation}
 \A(\lambda,z)=(\lambda_1z_1,\ldots,\lambda_nz_n)
  \end{equation}
 \begin{lemma}\label{lemma:Poissonmap}
 The action of $(\Cc^\ast)^n$ on $M$ is Poisson i.e.  for any $\lambda\in(\Cc^\ast)^n$ the linear map $\A_{\lambda}:M\to M$ defined by
 \[z\mapsto(\lambda_1z_1,\ldots,\lambda_nz_n)\]
  is a Poisson map. 
 \end{lemma}
 \begin{proof}
 In real coordinates, we denote $\lambda=(\xi_0,\eta_0)$ and $z=(\xi,\eta)$. By this notation
 \[L_\lambda(z)=(\xi_0\xi-\eta_0\eta,\eta_0\xi+\xi_0\eta),\]
 where $\xi_0\xi$ stands for component-wise multiplication of these vectors. Similarly for $\xi_0\eta, \eta_0\xi$ and $\eta_0\xi$. 
 We need to check condition~\eqref{Poissoncondition1}. Clearly,
 \[\d_{(\xi,\eta)}\A_{\lambda}=\begin{pmatrix}D_{\xi_0}&-D_{\eta_0}\\D_{\eta_0}&D_{\xi_0}\end{pmatrix}\]
 Simple calculation shows
 \begin{align*}
 &{1\over4}(\d_{(\xi,\eta)}\A_{\lambda})\begin{pmatrix}D_\xi AD_\xi&D_\xi AD_\eta\\D_\eta AD_\xi&D_\eta AD_\eta\end{pmatrix}(\d_{(\xi,\eta)}\A_{\lambda})^t=\\
 &{1\over4}\begin{pmatrix}D_{(\xi_0\xi-\eta_0\eta)} AD_{(\xi_0\xi-\eta_0\eta)}&D_{(\xi_0\xi-\eta_0\eta)} AD_{(\eta_0\xi+\xi_0\eta)}\\D_{(\eta_0\xi+\xi_0\eta)} AD_{(\xi_0\xi-\eta_0\eta)}&D_{(\eta_0\xi+\xi_0\eta)} AD_{(\eta_0\xi+\xi_0\eta)}\end{pmatrix}=\pi_M(\A_{\lambda}(\xi,\eta))
 \end{align*}
 \end{proof}
 
 Consider the subgroup $G$ of $(\Cc^\ast)^n$ defined by,
 $$ G:= \{(\lambda_1,\ldots,\lambda_n)\in(\C^ \ast)^n\,:\, 
\vert \lambda_i\vert  = \vert \lambda_j\vert ,\; \forall \, i,j\in\alpha,\, \forall\,\alpha=1,\ldots, p\,\}$$ 
The Poisson action of $(\Cc^\ast)^n$ on $M$ restricts to a Poisson action of  $G$ on $M$. 
Clearly, this action is proper. 

By Theorem~\ref{thm:singular} the quotient space $M/G$ is a Poisson stratified space.  The quotient space $M/G$ can be identified by $\Gamma_{\nund}$.  The identification is obtained
via the map 
\[(x^1,\ldots,x^p):\Cc^{n_1}\backslash\{{\bf 0}\}\times\ldots\times\Cc^{n_p}\backslash\{{\bf 0}\}\to\Delta^{n_1-1}\times\ldots\times\Delta^{n_p-1},\]
where $x^\alpha:\Cc^{n_\alpha}\backslash\{{\bf 0}\}\to\Delta^{n_\alpha-1}$ defined by
\[x^\alpha(z_1^\alpha,\ldots,z_{n_\alpha}^\alpha)=({|z_1^\alpha|^2\over |z_1^\alpha|^2+\ldots+|z_{n_\alpha}^\alpha|^2},\ldots,{|z_{n_\alpha}^\alpha |^2\over |z_1^\alpha|^2+\ldots+|z_{n_\alpha}^\alpha |^2})\]

Above and in the sequal ${\bf 0}=(0,\ldots,0)$ denotes a zero vector, while $0$ stands for a zero scalar. The strata of $\Gamma_{\nund}=M/G$ i.e. the $(H)$-orbit type reduced  space with $H$ being a subgroup of $G$, see Theorem~\ref{thm:singular}, are identified with the faces of $\Gamma_{\nund}$.  Let $e$ be the identity element of $G$, then the $(\{e\})$-orbit type reduced space is $\Gamma_{\nund}^\circ$, the interior of $\Gamma_{\nund}$. 

For  any $\alpha=1,\ldots,p$, $i\in\alpha$  
\begin{align*}
\d x^\alpha_i=&-2x^\alpha_i\frac{1}{r^\alpha}({\bf 0},\ldots,\xi^\alpha,\ldots,{\bf 0},{\bf 0},\ldots,\eta^\alpha,\ldots,{\bf 0})^t\\
&+\frac{2}{r^\alpha}({\bf 0},\ldots,(0,\ldots,\xi^\alpha_i,\ldots,0),\ldots,{\bf 0},\ldots,(0,\ldots,\eta^\alpha_i,\ldots,0),\ldots,{\bf0})^t,
\end{align*}
where $r^\alpha=|z_1^\alpha|^2+\ldots+|z_{n_\alpha}^\alpha|^2$. Let $\beta=1,\ldots,p$ and $j\in\beta$ also,  by definition, see~\eqref{quotientPoisson}, we have
\begin{align*}
\{x^\alpha_i,x^\beta_j\}_{\Gamma_{\nund}^\circ}&={1\over4}(\d x_i^\alpha)^t\begin{pmatrix}D_\xi AD_\xi&D_\xi AD_\eta\\D_\eta AD_\xi&D_\eta AD_\eta\end{pmatrix}(\d x^\beta_j)\\
&=(W^\alpha)^t\,\begin{pmatrix}D_{\xi^\alpha} A^{\alpha,\beta}D_{\xi^\beta}&D_{\xi^\alpha} A^{\alpha,\beta}D_{\eta^\beta}\\D_{\eta^\alpha} A^{\alpha,\beta}D_{\xi^\beta}&D_{\eta^\alpha} A^{\alpha,\beta}D_{\eta^\beta}\end{pmatrix}\,W^\beta,
\end{align*}
where
$$ W^\alpha:=\frac{1}{r^\alpha} \left[-x^\alpha_i  \xi^\alpha + (0,\ldots,\xi^\alpha_i,\ldots,0),
-x^\alpha_i \eta^\alpha + (0\ldots,\eta^\alpha_i,\ldots,0) 
\right]$$
A strait forward calculations show:
\begin{align*}
\{x^\alpha_i,x^\beta_j\}_{\Gamma_{\nund}^\circ}=(\pi_A)^{\alpha,\beta}_{i,j},
\end{align*}
where $\pi_A$ is defined at \eqref{poissonpmg}. This shows that $\{x^\alpha_i,x^\beta_j\}_{\Gamma_{\nund}^\circ}$  is the same Poisson structure on $\Gamma_{\nund}^\circ$ that was considered in Section \eqref{pmg} . The same holds for all the faces of $\Gamma_{\nund}$, which justifies Remark~\ref{remark:poissonreduction}. 

Notice that in our case the condition \eqref{condition:poisson1} is an algebraic equality which holds on $\Gamma_{\nund}^\circ$. The equality \eqref{condition:poisson1} is invariant w.r.t. multiplication of $x_i,x_j,x_k$ with constant numbers. Hence our algebraic equality must hold  on the open subset $\mathbb{R}^n_+$, which in turn yields that it is satisfied all over $\mathbb{R}^n$, i.e., $\pi_A$ is actually a Poisson structure on $\mathbb{R}^n$.


\section{Examples}
\label{examples}
It is possible to fully classify the dynamics of $2{\rm D}$ and $3{\rm D}$ conservative polymatrix replicator systems,
but in this section we just briefly describe two examples of $3{\rm D}$ polymatrix replicators.

\subsection*{First Example}
Consider the signature $\nund=(2,2,2)$, take the skew sym\-metric matrix 
$$ A_0=\left[ \begin{array}{rrrrrr}
 0 & -1 & 0 & \frac{1}{2} & 0 & 1 \\
 1 & 0 & 0 & -\frac{1}{2} & -1 & \frac{1}{2} \\
 0 & 0 & 0 & 0 & \frac{1}{2} & -1 \\
 -\frac{1}{2} & \frac{1}{2} & 0 & 0 & 0 & 0 \\
 0 & 1 & -\frac{1}{2} & 0 & 0 & -\frac{1}{2} \\
 -1 & -\frac{1}{2} & 1 & 0 & \frac{1}{2} & 0
\end{array} \right] \;,$$
and the  point $p=\left(\frac{7}{4},\frac{3}{4},\frac{5}{4},1,1,1\right)$ such that
$A_0\,p= \left(\frac{3}{4},\frac{3}{4},-\frac{1}{2},-\frac{1}{2},-\frac{3}{8},-\frac{3}{8}\right)$.
Consider the matrix  $A=A_0\,D$, where $D=\diag\left(\frac{5}{2},\frac{5}{2},\frac{9}{4},\frac{9}{4},2,2\right)$.
This matrix is 
$$ A=\left[  
\begin{array}{rrrrrr}
 0 & -{5}/{2} & 0 &  {9}/{8} & 0 & 2 \\
  {5}/{2} & 0 & 0 & -{9}/{8} & -2 & 1 \\
 0 & 0 & 0 & 0 & 1 & -2 \\
 -{5}/{4} & {5}/{4} & 0 & 0 & 0 & 0 \\
 0 & {5}/{2} & -{9}/{8} & 0 & 0 & -1 \\
 -{5}/{2} & -{5}/{4} & {9}/{4} & 0 & 1 & 0
\end{array} \right] \;.$$
By remark~\ref{skew:pmg:model}  $((2,2,2),A)$ is a conservative  polymatrix game.
The phase space of the associated  replicator system is the cube
$$\Gamma_{(2,2,2)}= \Delta^1 \times\Delta^1 \times\Delta^1 \equiv [0,1]^3\;.$$
In the model $[0,1]^3$, the equilibrium point $q=D^{-1}p$ has coordinates
$q=\left(\frac{7}{10},\frac{5}{9},\frac{1}{2}\right)$, and hence is an interior point.
The line through $q$ with direction $v=\left(\frac{6}{5},-\frac{4}{9},-1\right)$
intersects the cube $[0,1]^3$ along the set $\Sigma$ of equilibria of this replicator system.
This set $\Sigma$ is a line segment joinning two points in the faces $\{x=1\}$ and $\{z=1\}$.
To compute the symplectic foliation of $]0,1[^3$ consider the matrix
$$ E=\left[\begin{array}{rrrrrr}
1 & -1 & 0 & 0 & 0 & 0 \\
0 & 0 & 1 & -1 & 0 & 0  \\
0 & 0 & 0 & 0 & 1 & -1  \\
\end{array} \right]$$
and define 
$B=-E\,A_0 E^ t$.
A simple calculation gives
$$B=\left[
\begin{array}{rrr}
 0 & 1 & -\frac{1}{2} \\
 -1 & 0 & -\frac{3}{2} \\
 \frac{1}{2} & \frac{3}{2} & 0
\end{array}
\right]\;.$$
\begin{figure}[h] 
\begin{center}
 \includegraphics[width=12.5cm]{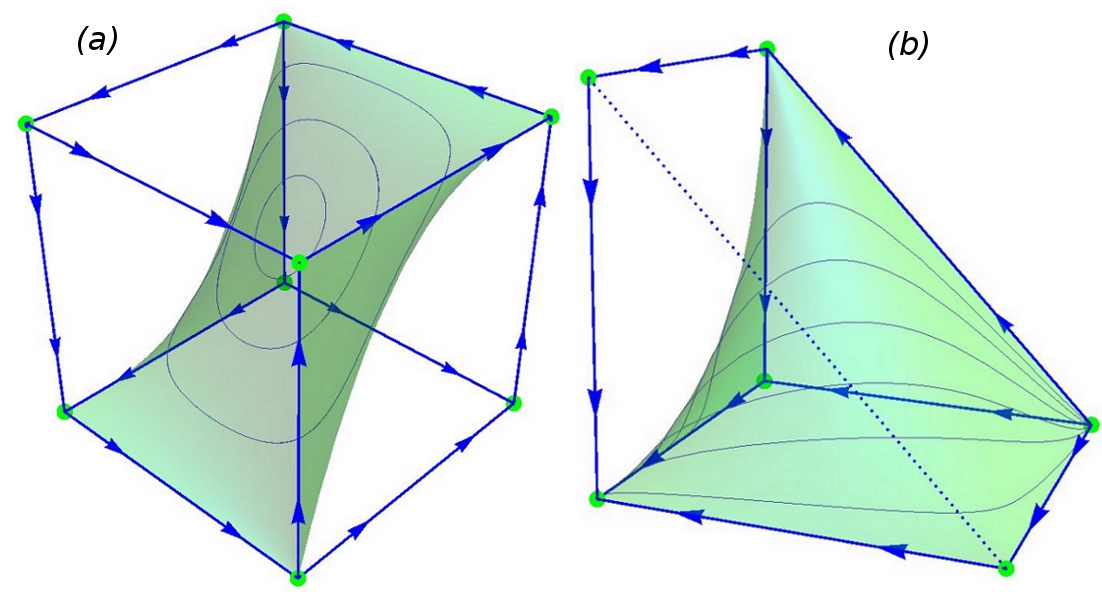}
\end{center}
\caption{Phase portraits of $3{\rm D}$ polymatrix replicators} \label{exs}
\end{figure}

The vector $w= \left(-\frac{3}{2},\frac{1}{2},1\right)$ is orthogonal to the space spanned by the columns of $B$.
The symplectic leaves of the constant Poisson structure on $\Rr^ 3$ defined by the skew symmetric matrix $B$ are
the planes orthogonal to $w$.
Thus, if we consider  the Poisson diffeomorphism
$\phi: \Rr^3 \to  ]0,1[^3$,
$$\phi(u_1,u_2,u_3)=\left( \frac{e^{u_1}}{1+e^{u_1}},
\frac{e^{u_2}}{1+e^{u_2}}, \frac{e^{u_3}}{1+e^{u_3}}  \right)\;,$$
the symplectic leaves on $ ]0,1[^3$ are  the $\phi$ images of these planes.
Inverting the map $\phi$, the symplectic leaves are given by the equations
\begin{align*}
&\left(\frac{x}{1-x}\right)^{-3/2}\left( \frac{y}{1-y} \right)^{1/2}\left( \frac{z}{1-z} \right)=e^c \\
& \Leftrightarrow\quad
(1-x)^{3/2} y^{1/2} z = e^c x^{3/2} (1-y)^{1/2} (1-z)\;,
\end{align*}
with $c\in\Rr$. Let $U_+$, resp. $U_-$, be the union of the faces
$\{x=1\}$, $\{y=0\}$, $\{z=0\}$, resp. $\{x=0\}$, $\{y=1\}$, $\{z=1\}$. On the interiors of these two open subsets of the cube's boundary the equation above is never satisfied. Therefore the closure of every symplectic leaf intersects the cube's boundary 
along the closed curve $C=\partial U_+ = \partial U_-\subset \partial [0,1]^3$. Because $\Sigma$ intersects both $U_-$ and $U_+$,
it follows that every symplectic leaf must intersect  $\Sigma$,
hence having a unique equilibrium. The orbits of our polymatrix replicator foliate each symplectic leaf into closed curves around that equilibrium point. We can also check that
$C$ is a heteroclinic cycle of the vector
field $X_{(2,2,2),A}$. See Figure~\ref{exs}(a).

\subsection*{Second Example} Consider the signature $\nund=(3,2)$, take the skew sym\-metric matrix 
$$ A_0=\left[ \begin{array}{rrrrr}
0 & 0 & \frac{1}{2} & \frac{1}{2} & -1 \\
 0 & 0 & -\frac{1}{2} & \frac{1}{2} & -\frac{1}{2} \\
 -\frac{1}{2} & \frac{1}{2} & 0 & 1 & \frac{1}{2} \\
 -\frac{1}{2} & -\frac{1}{2} & -1 & 0 & 0 \\
 1 & \frac{1}{2} & -\frac{1}{2} & 0 & 0 \\
\end{array} \right] \;,$$
and the  point $p=\left(\frac{9}{10},-\frac{8}{5},\frac{1}{2},0,1\right)$ such that
$A_0\,p= \left(-\frac{3}{4},-\frac{3}{4},-\frac{3}{4},-\frac{3}{20},-\frac{3}{20}\right)$.
Consider the matrix  $A=A_0\,D$, where $D=\diag\left(-\frac{1}{5},-\frac{1}{5},-\frac{1}{5},1,1\right)$.
This matrix is 
$$ A=\left[  
\begin{array}{rrrrr}
 0 & 0 & -\frac{1}{10} & \frac{1}{2} & -1 \\
 0 & 0 & \frac{1}{10} & \frac{1}{2} & -\frac{1}{2} \\
 \frac{1}{10} & -\frac{1}{10} & 0 & 1 & \frac{1}{2} \\
 \frac{1}{10} & \frac{1}{10} & \frac{1}{5} & 0 & 0 \\
 -\frac{1}{5} & -\frac{1}{10} & \frac{1}{10} & 0 & 0 \\
\end{array} \right] \;.$$
By remark~\ref{skew:pmg:model} $((3,2),A)$ is a conservative
polymatrix game.
The phase space of the associated replicator system is the prism
$$\Gamma_{(3,2)}= \Delta^2 \times\Delta^1  \equiv 
\{\,(x,y,z)\,:\, 0\leq x,y,z\leq 1,\; x+y\leq 1\,\} =:P \;.$$

In the model $P\subset\Rr^3$ the equilibrium point $q=D^{-1}p$ has coordinates
$q=\left(-\frac{9}{2},8,0\right)$, and hence is not interior to $P$.
The line of equilibria goes through $q$ with direction $v=\left(-\frac{5}{2},5,-1\right)$ and does not
intersect  the prism $P$ .
To compute the symplectic foliation of $P^\circ$ consider the matrix
$$ E=\left[\begin{array}{rrrrr}
1 & -1 & 0 & 0 & 0  \\
1 & 0 & -1 & 0 & 0  \\
0 & 0 & 0 & 1 & -1   \\
\end{array} \right]$$
and define 
$B=-E\,A_0 E^ t$.
A simple calculation gives
$$B=\left[
\begin{array}{rrr}
 0 & 1 & -\frac{1}{2} \\
 -1 & 0 & -1 \\
 \frac{1}{2} & 1 & 0 
\end{array}
\right]\;.$$
The vector $w= \left(-1,\frac{1}{2},1\right)$ is orthogonal to the space spanned by the columns of $B$.
The symplectic leaves of the constant Poisson structure on $\Rr^ 3$ defined by the skew symmetric matrix $B$ are
the planes orthogonal to $w$.
Thus, if we consider  the Poisson diffeomorphism
$\phi: \Rr^3 \to  P^\circ$,
$$\phi(u_1,u_2,u_3)=\left( \frac{e^{u_1}}{1+e^{u_1}+e^{u_2}},
\frac{e^{u_2}}{1+e^{u_1}+e^{u_2}}, \frac{e^{u_3}}{1+e^{u_3}}  \right)\;,$$
the symplectic leaves on $P^\circ$ are  the $\phi$ images of these planes.
Inverting the map $\phi$, the symplectic leaves are given by the equations
\begin{align*}
&\left(\frac{x}{1-x-y}\right)^{-1}\left( \frac{y}{1-x-y} \right)^{1/2}\left( \frac{z}{1-z} \right)=e^c \\
& \Leftrightarrow\quad
(1-x-y)^{1/2} y^{1/2} z = e^c x    (1-z)\;,
\end{align*}
with $c\in\Rr$. Let $U_+$, resp. $U_-$, be the union of the faces
$\{x+y=1\}$, $\{y=0\}$, $\{z=0\}$, resp. $\{x=0\}$, $\{z=1\}$. On the interiors of these two open subsets of the prism's boundary the equation above is never satisfied. Therefore the closure of every symplectic leaf intersects the prism's boundary 
along the closed curve $C=\partial U_+ = \partial U_-\subset \partial P$.  The points $r=(1,0,0)$ and $s=(0,0,1)$ on $C$
are respectively a global  repeller and  a global sink of 
the polymatrix replicator, and every symplectic leaf is foliated into orbits flowing from the repeller $r$ to the sink $s$. The closed curve 
$C$ is also the union of two heteroclinic chains from $r$ to $s$.
See Figure~\ref{exs}(b).
Note that this dynamical behaviour does not contradict the Hamiltonian character of the system because the area of each symplectic leaf is infinite.

\section*{Acknowledgements}
Both authors would like to thanks Rui Loja Fernandes for valuable suggestions and in particular for pointing them the example in  \cite{MR2555841}*{Section 2.5} .

The first author was supported by the Geometry and Mathematical Physics Project, 
FCT EXCL/MAT-GEO/0222/2012. The second author was supported by ``Funda\c{c}\~{a}o para a Ci\^{e}ncia e a Tecnologia'' through the Program POCI 2010 and the Project ``Randomness in Deterministic Dynamical Systems and Applications'' (PTDC-MAT-105448-2008).


\begin{bibdiv}
\begin{biblist}

\bib{MR765812}{article}{
   author={Akin, Ethan},
   author={Losert, Viktor},
   title={Evolutionary dynamics of zero-sum games},
   journal={J. Math. Biol.},
   volume={20},
   date={1984},
   number={3},
   pages={231--258},
   issn={0303-6812},
   review={\MR{765812 (86g:92024a)}},
   doi={10.1007/BF00275987},
}

\bib{MR1643678}{article}{
   author={Duarte, Pedro},
   author={Fernandes, Rui L.},
   author={Oliva, Waldyr M.},
   title={Dynamics of the attractor in the Lotka-Volterra equations},
   journal={J. Differential Equations},
   volume={149},
   date={1998},
   number={1},
   pages={143--189},
   issn={0022-0396},
   review={\MR{1643678 (99h:34075)}},
   doi={10.1006/jdeq.1998.3443},
}

\bib{MR2178041}{book}{
   author={Dufour, Jean-Paul},
   author={Zung, Nguyen Tien},
   title={Poisson structures and their normal forms},
   series={Progress in Mathematics},
   volume={242},
   publisher={Birkh\"auser Verlag, Basel},
   date={2005},
   pages={xvi+321},
   isbn={978-3-7643-7334-4},
   isbn={3-7643-7334-2},
}

\bib{MR1738431}{book}{
   author={Duistermaat, J. J.},
   author={Kolk, J. A. C.},
   title={Lie groups},
   series={Universitext},
   publisher={Springer-Verlag, Berlin},
   date={2000},
   pages={viii+344},
   isbn={3-540-15293-8},
   review={\MR{1738431 (2001j:22008)}},
   doi={10.1007/978-3-642-56936-4},
}


\bib{MR723584}{article}{
   author={Eshel, I.},
   author={Akin, E.},
   title={Coevolutionary instability and mixed Nash solutions},
   journal={J. Math. Biol.},
   volume={18},
   date={1983},
   number={2},
   pages={123--133},
   issn={0303-6812},
   review={\MR{723584 (85d:92023)}},
   doi={10.1007/BF00280661},
}


\bib{MR2555841}{article}{
   author={Fernandes, Rui Loja},
   author={Ortega, Juan-Pablo},
   author={Ratiu, Tudor S.},
   title={The momentum map in Poisson geometry},
   journal={Amer. J. Math.},
   volume={131},
   date={2009},
   number={5},
   pages={1261--1310},
   issn={0002-9327},
   review={\MR{2555841 (2011f:53199)}},
   doi={10.1353/ajm.0.0068},
}


\bib{MR1393843}{article}{
   author={Hofbauer, Josef},
   title={Evolutionary dynamics for bimatrix games: a Hamiltonian system?},
   journal={J. Math. Biol.},
   volume={34},
   date={1996},
   number={5-6},
   pages={675--688},
   issn={0303-6812},
   review={\MR{1393843 (97h:92011)}},
   doi={10.1007/s002850050025},
}
	
\bib{MR1635735}{book}{
   author={Hofbauer, Josef},
   author={Sigmund, Karl},
   title={Evolutionary games and population dynamics},
   publisher={Cambridge University Press},
   place={Cambridge},
   date={1998},
   pages={xxviii+323},
   isbn={0-521-62365-0},
   isbn={0-521-62570-X},
   review={\MR{1635735 (99h:92027)}},
}

\bib{MR633014}{article}{
   author={Hofbauer, Josef},
   title={On the occurrence of limit cycles in the Volterra-Lotka equation},
   journal={Nonlinear Anal.},
   volume={5},
   date={1981},
   number={9},
   pages={1003--1007},
   issn={0362-546X},
   review={\MR{633014 (83c:92063)}},
   doi={10.1016/0362-546X(81)90059-6},
}
\bib{MR0392000}{article}{
   author={Howson, Joseph T., Jr.},
   title={Equilibria of polymatrix games},
   journal={Management Sci.},
   volume={18},
   date={1971/72},
   pages={312--318},
   issn={0025-1909},
   review={\MR{0392000 (52 \#12818)}},
}

\bib{MR1048350}{book}{
   author={Perelomov, A. M.},
   title={Integrable systems of classical mechanics and Lie algebras. Vol.
   I},
   note={Translated from the Russian by A. G. Reyman [A. G. Re\u\i man]},
   publisher={Birkh\"auser Verlag, Basel},
   date={1990},
   pages={x+307},
   isbn={3-7643-2336-1},
   review={\MR{1048350 (91g:58127)}},
   doi={10.1007/978-3-0348-9257-5},
}

\bib{MR1869601}{book}{
   author={Pflaum, Markus J.},
   title={Analytic and geometric study of stratified spaces},
   series={Lecture Notes in Mathematics},
   volume={1768},
   publisher={Springer-Verlag, Berlin},
   date={2001},
   pages={viii+230},
   isbn={3-540-42626-4},
   review={\MR{1869601 (2002m:58007)}},
}




\bib{MR1024957}{article}{
   author={Quintas, L. G.},
   title={A note on polymatrix games},
   journal={Internat. J. Game Theory},
   volume={18},
   date={1989},
   number={3},
   pages={261--272},
   issn={0020-7276},
   review={\MR{1024957 (91a:90188)}},
   doi={10.1007/BF01254291},
}



\bib{MR800991}{article}{
   author={Redheffer, Ray},
   title={Volterra multipliers. I, II},
   journal={SIAM J. Algebraic Discrete Methods},
   volume={6},
   date={1985},
   number={4},
   pages={592--611, 612--623},
   issn={0196-5212},
   review={\MR{800991 (87j:15037a)}},
   doi={10.1137/0606059},
}

\bib{MR1027969}{article}{
   author={Redheffer, Ray},
   title={A new class of Volterra differential equations for which the
   solutions are globally asymptotically stable},
   journal={J. Differential Equations},
   volume={82},
   date={1989},
   number={2},
   pages={251--268},
   issn={0022-0396},
   review={\MR{1027969 (91f:34058)}},
   doi={10.1016/0022-0396(89)90133-2},
}

\bib{MR741270}{article}{
   author={Redheffer, Ray},
   author={Walter, Wolfgang},
   title={Solution of the stability problem for a class of generalized
   Volterra prey-predator systems},
   journal={J. Differential Equations},
   volume={52},
   date={1984},
   number={2},
   pages={245--263},
   issn={0022-0396},
   review={\MR{741270 (85k:92068)}},
   doi={10.1016/0022-0396(84)90179-7},
}

\bib{MR646217}{article}{
   author={Redheffer, Ray},
   author={Zhou, Zhi Ming},
   title={Global asymptotic stability for a class of many-variable Volterra
   prey-predator systems},
   journal={Nonlinear Anal.},
   volume={5},
   date={1981},
   number={12},
   pages={1309--1329},
   issn={0362-546X},
   review={\MR{646217 (83h:92074)}},
   doi={10.1016/0362-546X(81)90108-5},
}

\bib{MR644963}{article}{
   author={Redheffer, Ray},
   author={Zhou, Zhi Ming},
   title={A class of matrices connected with Volterra prey-predator
   equations},
   journal={SIAM J. Algebraic Discrete Methods},
   volume={3},
   date={1982},
   number={1},
   pages={122--134},
   issn={0196-5212},
   review={\MR{644963 (83m:15020)}},
   doi={10.1137/0603012},
}


\bib{MR1189803}{book}{
   author={Volterra, Vito},
   title={Le\c cons sur la th\'eorie math\'ematique de la lutte pour la vie},
   language={French},
   series={Les Grands Classiques Gauthier-Villars. [Gauthier-Villars Great
   Classics]},
   note={Reprint of the 1931 original},
   publisher={\'Editions Jacques Gabay, Sceaux},
   date={1990},
   pages={vi+215},
   isbn={2-87647-066-7},
   review={\MR{1189803 (93k:92011)}},
}


\bib{smith82}{book}{
  author = {Smith, John Maynard},
  title = {Evolution and the Theory of Games},
  publisher = {Cambridge University Press},
  date ={1982}
}





		
		
\end{biblist}
\end{bibdiv}
\end{document}